\documentclass[reqno, 12pt]{amsart}
\usepackage{amsfonts}
\usepackage{mathrsfs}
\usepackage{amssymb,url}
\usepackage{graphicx}

\date{}

\allowdisplaybreaks[4]
\theoremstyle{plain}
\newtheorem{thm}{Theorem}[section]

\newtheorem{lem}[thm]{Lemma}
\newtheorem{Lem1}{Lemma}
\newtheorem{Thm}{Theorem}

\newtheorem{pro}[thm]{Proposition}
\newtheorem{example}{Example}
\newtheorem{cla}{Claim}

\newtheorem*{Rem}{Remark}
\theoremstyle{definition}
\newtheorem{definition}{Definition}

\usepackage[colorlinks,linkcolor=blue,urlcolor=red,linktocpage]{hyperref}

\numberwithin{equation}{section}

\begin{document}

\title
[~~~]{ A new approach to Steiner symmetrization of coercive convex
functions}

\author[Youjiang Lin]{Youjiang Lin}
\address{School of Mathematical Sciences, Peking University, Beijing, 100871, China;
Department of Mathematics, Department of Mathematics, Shanghai University, Shanghai, 200444, China}
 \email{\href{mailto: YOUJIANG LIN
<lxyoujiang@126.com>}{lxyoujiang@126.com}}
\author[Gangsong Leng]{Gangsong Leng}
\address{Department of Mathematics, Shanghai University, Shanghai, 200444, China} \email{\href{mailto:
Gangsong Leng <gleng@staff.shu.edu.cn>}{gleng@staff.shu.edu.cn} }
\author[Lujun Guo]{Lujun Guo}
\address{Department of Mathematics, Shanghai University, Shanghai, People's Republic of China 200444} \email{\href{mailto:
Lujun Guo <lujunguo0301@163.com>}{lujunguo0301@163.com} }

\begin{abstract} In this paper, a new approach of defining Steiner symmetrization of
coercive convex functions is proposed and some fundamental
properties of the new Steiner symmetrization are proved.
 Further, using the new Steiner symmetrization, we give a  different approach to prove
 a functional version of the Blaschke-Santal\'{o} inequality due to
Ball \cite{Ba86}.
\end{abstract}



\thanks{The authors would like to acknowledge the support from China Postdoctoral Science
Foundation Grant 2013M540806, National Natural Science Foundation of China under grant
11271244 and National Natural Science Foundation of China under grant 11271282 and the
973 Program 2013CB834201.}

\maketitle

\section{Introduction}
The purpose of this paper is to introduce a new way of defining
Steiner symmetrization for coercive convex functions, and to explore
its applications. Our new definition is motivated by and can be
regarded as an improvement of a functional Steiner symmetrization of
\cite{Ar04}. In particular, our new definition has a key property:
the invariance of integral, which is not true for the definition of
\cite{Ar04}. Moreover, our definition provides a new approach to the
familiar functional Steiner symmetrization (see \cite{Br74,Bu97}),
but we do not use geometric Steiner symmetrization and our approach
is more suitable for certain functional problems.

Steiner symmetrization was invented by Steiner \cite{St38} to prove
the isoperimetric inequality. For over 160 years Steiner
symmetrization has been a fundamental tool for attacking problems
regarding isoperimetry and related geometric inequalities
\cite{Ga02,Ga06,St38,Ta93}. Steiner symmetrization appears in the
titles of dozens of papers (see e.g.
\cite{Bi03,Bi11,Bo87,Bu97,Ci05,Fa76,Ga83,Kl03,Kl05,Lo86,Ma86,Mc67,Sc98})
and plays a key role in recent work such as
\cite{Ha09,Lu10,Sc05,Sc06}.

Steiner symmetrization is a type of rearrangement. In the 1970s,
interest in rearrangements was renewed, as mathematicians began to
look for geometric proofs of functional inequalities. Rearrangements
were generalized from smooth or convex bodies to measurable sets and
to functions in Sobolev spaces. Functional Steiner symmetrization,
as a kind of important rearrangement of functions, has been studied
in \cite{Ar04,Br74,Bu97,Bu09,Ci04,Ci06,Fo10}. In \cite{Br74},
Brascamp, Lieb, and Luttinger established that the spherical
symmetrization of a nonnegative function can be approximated in
$L^p(\mathbb{R}^{n+1})$ by a sequence of
 Steiner symmetrizations and rotations. In \cite{Bu97}, Burchard
 proved that Steiner symmetrization is continuous in $W^{1,p}(\mathbb{R}^{n+1})$,
$1\leq p<\infty$, for every dimension $n\geq 1$, in the sense that
$f_k\rightarrow f$ in $W^{1,p}$ implies $S f_k\rightarrow S f$ in
$W^{1,p}$. In \cite{Fo10}, Fortier gave a thorough review and
exposition of results regarding approximating the symmetric
decreasing rearrangement by polarizations and Steiner
symmetrizations.

 For a nonnegative measurable function $f$, the familiar definition of its Steiner
symmetrization (see \cite{Br74, Bu97, Bu09, Fo10}) is defined as following:
\begin{definition}\label{c25}
For a measurable function $f:\mathbb{R}^n\rightarrow\mathbb{R}^+$,
let $m$ denote the Lebesgue measure, if $m([f>t])<+\infty$ for all
$t>0$, then its Steiner symmetrization is defined as
\begin{eqnarray}\label{1} \bar{S}_u f(x)=\int_{0}^{\infty}\mathcal
{X}_{S_u E(t)}(x)dt,
\end{eqnarray}
where $S_u E(t)$ is the Steiner symmetrization of the level set
$E(t):=\{x\in\mathbb{R}^n: f(x)>t\}$ about $u^{\perp}$ and $\mathcal
{X}_{A}$ denotes the characteristic function of set $A$.
\end{definition}
During the study of the analogy between convex bodies and
log-concave functions, Artstein-Klartag-Milman in \cite{Ar04}
defined another functional Steiner transformation as follows:
\begin{definition}\label{c26}
For a coercive convex function $f: \mathbb{R}^n\rightarrow
\mathbb{R}\cup\{+\infty\}$ and a hyperplane $H=u^{\bot}$ ($u\in
S^{n-1}$) in $\mathbb{R}^n$, for
 any $x=x^{\prime}+tu$, where
$x^{\prime}\in H$ and $t\in \mathbb{R}$, we define the {\it Steiner
symmetrization} $\widetilde{S}_u f$ of $f$ about $H$ by
\begin{eqnarray}
(\widetilde{S}_u f)(x)=\inf_{t_1+t_2=t}[\frac{1}{2}
f(x^{\prime}+2t_1u)+\frac{1}{2}f(x^{\prime}-2t_2u)].
\end{eqnarray}
\end{definition}
 In this paper, we introduce a new way of defining the
functional Steiner symmetrization for coercive convex functions.
 A function $f:\mathbb{R}^n\rightarrow
\mathbb{R}\cup \{+\infty\}$, not identically $+\infty$, is called
{\it convex} if $$f(\alpha x+(1-\alpha)y)\leq \alpha
f(x)+(1-\alpha)f(y)$$ for all $x$, $y\in\mathbb{R}^n$ and for
$0\leq\lambda\leq1$. A convex function $f:\mathbb{R}^n\rightarrow
\mathbb{R}\cup \{+\infty\}$ is called {\it coercive} if
$\lim_{|x|\rightarrow +\infty}f(x)=+\infty$.

\begin{definition}\label{15}For
a coercive convex function $f: \mathbb{R}^n\rightarrow
\mathbb{R}\cup\{+\infty\}$ and a hyperplane $H=u^{\bot}$ ($u\in
S^{n-1}$) in $\mathbb{R}^n$, for
 any $x=x^{\prime}+tu$, where
$x^{\prime}\in H$ and $t\in \mathbb{R}$, we define the {\it Steiner
symmetrization} $S_u f$ of $f$ about $H$ by
\begin{eqnarray}\label{3}
(S_u f)(x)=\sup_{\lambda\in [0,1]}\inf_{t_1+t_2=t}[\lambda
f(x^{\prime}+2t_1u)+(1-\lambda)f(x^{\prime}-2t_2u)].
\end{eqnarray}
\end{definition}

Our definition $S_u f$ is motivated by and can be regarded as an
improvement of $\widetilde{S}_u f$ in Definition \ref{c26}.
 When compared with $\bar{S}_uf$ in Definition \ref{c25}, our definition symmetrizes a
 parabola-like (one-dimension) cure once at a time instead of
 symmetrizing the level set as in $\bar{S}_uf$.

The rest of the paper is organized as follows. In Section 2, we
explore the analogy between convex bodies and coercive convex
functions using our new definition (see Table \ref{a59}).
\begin{table}[!ht] \label{a59}{\small Table 1. A contrast between convex bodies and coercive
convex functions on Steiner symmetrization}\vspace{0mm}
\begin{center}\small\begin{tabular}{|c|p{160pt}|p{170pt}|} \hline
~& \multicolumn{1}{|c|}{Convex bodies}& \multicolumn{1}{|c|}{Coercive convex functions} \\
\hline 1& For a convex body $K$, $S_u K$ is still a convex body and symmetric about $u^{\perp}$.
&For a coercive convex function $f$,
 $S_u f$ is still a coercive convex function and symmetric about $u^{\perp}$.\\
\hline 2&$Vol_n(S_u K)=Vol_n(K)$. & $\int_{\mathbb{R}^n}\exp(-S_u f)=\int_{\mathbb{R}^n}\exp(-f)$.\\
\hline 3& $K$ can be transformed into an unconditional body using
$n$ Steiner symmetrizations.&$f$ can be transformed into an
unconditional function using $n$ Steiner symmetrizations.\\
\hline 4&For any convex bodies $K_1\subset K_2$, then $S_u K_1\subset S_u K_2$.
&For any coercive convex functions $f_1\leq f_2$, then $S_u f_1\leq S_u f_2$.\\
\hline 5&If $K$ is a symmetric about $z$, then $S_u K$ is symmetric
about
$z|u^{\perp}$.&If $f$ is even about $z$, then $S_u f$ is even about $z|u^{\perp}$.\\
\hline 6&If the sequence $\{K_i\}$ converges in the Hausdorff metric to $K$,
then the sequence $\{S_u K_i\}$ will converge to $S_u K$.&If the sequence $\{\exp(-f_i)\}$
converges in the $L^p$ distance to $\exp(-f)$, then the sequence $\{\exp(-S_u f_i)\}$ will converge to $\exp(-S_u f)$.\\
\hline 7&There is a sequence of directions $\{u_i\}$ so that the
sequence of convex bodies $K_i=S_{u_i}\dots S_{u_1} K$ converges to
the ball with the same volume as $K$.&There is a sequence of
directions $\{u_i\}$ so that the sequence of log-concave functions
$\exp(-f_i)$, where $f_i=S_{u_i}\dots S_{u_1} f$, converges to a
radial function
with the same integral as $\exp(-f)$.\\
\hline
\end{tabular}
\end{center}
\end{table}
In Section 3, we will elaborate on the relations between Definition
\ref{15} and Definitions \ref{c25}, \ref{c26}. In Section 4, we give
a completely different approach to prove
 a functional version of the Blaschke-Santal\'{o} inequality due to
Ball \cite{Ba86}.

\section{ The functional Steiner symmetrization}

We first study the one-dimensional case. In Definition \ref{15},
when $n=1$, $S^0=\{-1, 1\}$ and $H=\{0\}$, it is clear that $(S_1
f)(x)=(S_{-1}f)(x)$ for any $x\in\mathbb{R}$. Let $Sf$ denote
Steiner symmetrization of one-dimensional function, then
\begin{eqnarray}\label{c16}
Sf(x)=\sup_{\lambda\in [0,1]}\inf_{x_1+x_2=x}[\lambda
f(2x_1)+(1-\lambda)f(-2x_2)].
\end{eqnarray}

\begin{Thm}\label{b10}If $f: \mathbb{R}\rightarrow \mathbb{R}\cup\{+\infty\}$ is a
coercive convex function, then  $Sf(x)$ is a coercive even convex
function and for any $s\in\mathbb{R}$,
\begin{eqnarray}\label{a2}Vol_1 ([f\leq s])=Vol_1([Sf\leq
s]),\end{eqnarray} where $[f\leq s]=\{x\in\mathbb{R}: f(x)\leq s\}$
denotes the sublevel set of $f$.
\end{Thm}
The following lemma is straightforward, and we omit its proof.
\begin{Lem1}\label{b14}
Let $f: \mathbb{R}\rightarrow \mathbb{R}$ be a coercive convex
function, then we have

(i) If $a=\inf f(t)$, then $a\in (-\infty,+\infty)$ and
$f^{-1}(a)=\{x\in\mathbb{R}: f(x)=a\}$ is a nonempty finite closed
interval $[\mu,\nu]$, where $\mu$ may equal to $\nu$.

(ii) $f(t)$ is strictly decreasing on the interval $(-\infty,\mu]$
and strictly increasing on the interval $[\nu,+\infty)$.

(iii) If $f(c)=f(d)$ and $c<d$, then $\mu<d$ and $c<\nu$.

(iv)  For $c$ and $d$ given in (iii), we have the right derivative
$f^{\prime}_{r}(d)\geq 0$ for $f$ is increasing on $[\mu,+\infty)$,
we also have $f^{\prime}_{r}(c)\leq 0$ for $f$ is decreasing on
$(-\infty,\nu]$.

(v) For two intervals $[a, a+t_0]$ and $[b, b+t_0]$ with the same
length $t_0>0$, if $f(a)=f(a+t_0)$, then either $f(b)\geq f(a)$ or
$f(b+t_0)\geq f(a+t_0)$.
\end{Lem1}

\noindent{\it Proof of Theroem \ref{b10}.}
 First, we show that $Sf$ is even. For any  $x\in \mathbb{R}$, by (\ref{c16}), we have
\begin{eqnarray}\label{22}
Sf(-x)&=&\sup_{\lambda\in [0,1]}\inf_{x_2\in \mathbb{R}}[\lambda
f(-2x_2-2x)+(1-\lambda)f(-2x_2)]
\nonumber\\
&=&\sup_{\lambda\in [0,1]}\inf_{x_2\in \mathbb{R}}[\lambda
f(2x_2-2x)+(1-\lambda)f(2x_2)]
\nonumber\\
&=&\sup_{\lambda^{\prime}\in [0,1]}\inf_{x_2\in
\mathbb{R}}[\lambda^{\prime}f(2x_2)+(1-\lambda^{\prime})f(2x_2-2x)]\nonumber\\
&=&Sf(x),
\end{eqnarray}
which implies that $Sf$ is even.

 Let ${\rm dom}
f:=\{x\in\mathbb{R}^n: f(x)<+\infty\}$ denote the effective domain
of $f$. To prove the remaining part of the theorem, we shall
consider two cases: ${\rm dom} f =\mathbb{R}$ and ${\rm
dom}f\neq\mathbb{R}$.

{\bf Case (1) ${\rm dom} f =\mathbb{R}$.} There are two steps.

{\bf First Step.} We shall prove that $Sf(0)=\inf f$ and for any
$x>0$, there exists some $x^{\prime}\in\mathbb{R}$ such that
\begin{eqnarray}\label{26}
Sf(x)=f(x^{\prime})=f(x^{\prime}-2x).
\end{eqnarray}
Let $x=0$, by (\ref{c16}), we have
\begin{eqnarray}\label{c10}
Sf(0)&=&\sup_{\lambda\in[0,1]}\inf_{x_1+x_2=0}[\lambda
f(2x_1)+(1-\lambda) f(-2x_2)]\nonumber\\
&=&\inf_{x_1\in\mathbb{R}}f(2x_1)=\inf_{x\in\mathbb{R}} f(x).
\end{eqnarray}
For $x>0$, since $f$ is coercive and convex, there exists some
$x^{\prime}\in\mathbb{R}$ satisfying
\begin{eqnarray}\label{b6}
f(x^{\prime})=f(x^{\prime}-2x).
\end{eqnarray}
Indeed, let $f_x(x_1):=f(x_1)-f(x_1-2x)$, $a=\inf f$ and
$f^{-1}(a)=[\mu,\nu]$,  by Lemma \ref{b14}(ii), $f_x(x_1)<0$ if
$x_1<\mu$ and $f_x(x_1)>0$ if $x_1>\nu$. Since $f(x_1)$ and
$f(x_1-2x)$ are convex functions about $x_1\in\mathbb{R}$ and any
convex function is continuous on the interior of its effective
domain, thus $f_x(x_1)$ is continuous in $\mathbb{R}$. Therefore,
there exists some $x^{\prime}$ such that $f_x(x^{\prime})=0$.

Now we prove $Sf(x)=f(x^{\prime})$, where $x>0$ and $x^{\prime}$
satisfies equality (\ref{b6}). Let $G_x(\lambda)$ be a function
about $\lambda\in[0,1]$ defined as
\begin{eqnarray}
G_x(\lambda):=\inf_{x_1\in\mathbb{R}}[\lambda
f(2x_1)+(1-\lambda)f(2x_1-2x)].
\end{eqnarray}

For any $\lambda\in [0,1]$, choose $x_1=\frac{x^{\prime}}{2}$, we
have
\begin{eqnarray}
G_x(\lambda)\leq\lambda
f(x^{\prime})+(1-\lambda)f(x^{\prime}-2x)=f(x^{\prime}).
\end{eqnarray} Thus, $Sf(x)=\sup_{\lambda\in [0,1]}G_x(\lambda)\leq
f(x^{\prime})$.

On the other hand, we prove that there exists some $\lambda_0\in
[0,1]$ such that $G_x(\lambda_0)=f(x^{\prime})$. Since $f$ is a
convex function defined in $\mathbb{R}$ and by Theorem 1.5.2 in
\cite{Sc93}, both the right derivative $f^{\prime}_{r}$ and the left
derivative $f^{\prime}_{l}$ exist and $f^{\prime}_{l}\leq
f^{\prime}_{r}$.
\begin{cla}\label{d45} There exists some $\lambda_0\in[0,1]$ satisfying
\begin{eqnarray}\label{b1}\lambda_0f^{\prime}_{r}(x^{\prime})+(1-\lambda_0)f^{\prime}_{r}(x^{\prime}-2x)=0.
\end{eqnarray}
\end{cla}
\noindent{\it Proof of Claim \ref{d45}.} Since
$f(x^{\prime})=f(x^{\prime}-2x)$ and $x>0$, by Lemma \ref{b14}(iv),
we have $f_r^{\prime}(x^{\prime})\geq 0$ and
$f_r^{\prime}(x^{\prime}-2x)\leq 0$, thus
$f_r^{\prime}(x^{\prime})-f_r^{\prime}(x^{\prime}-2x)\geq 0$.

(i) If
 $f^{\prime}_{r}(x^{\prime})-f^{\prime}_{r}(x^{\prime}-2x)>0$,
then (\ref{b1}) can be obtained by choosing
\begin{eqnarray}
\lambda_0=\frac{-f^{\prime}_{r}(x^{\prime}-2x)}{f^{\prime}_{r}(x^{\prime})-f^{\prime}_{r}(x^{\prime}-2x)}.
\end{eqnarray}

(ii) If
$f^{\prime}_{r}(x^{\prime})-f^{\prime}_{r}(x^{\prime}-2x)=0$, then
$f_r^{\prime}(x^{\prime})=f_r^{\prime}(x^{\prime}-2x)= 0$, thus, for
any $\lambda_0\in [0,1]$, we can get (\ref{b1}). \hfill $\Box$\\

Choose a $\lambda_0$ satisfying (\ref{b1}), we define
\begin{eqnarray}\label{b7}\Phi_{\lambda_0}(x_1)=\lambda_0f(2x_1)+(1-\lambda_0)f(2x_1-2x).\end{eqnarray}

Since $f$ is a convex function, then $\Phi_{\lambda_0}$ is a convex
function about $x_1$. By (\ref{b1}), we have that the right
derivative and the left derivative of $\Phi_{\lambda_0}$ at
$x_1=\frac{x^{\prime}}{2}$ satisfy
\begin{eqnarray}\label{b2}
\Phi_{\lambda_0
r}^{\prime}(x_1)|_{x_1=\frac{x^{\prime}}{2}}=2\lambda_0f_r^{\prime}(x^{\prime})+2(1-\lambda_0)f_r^{\prime}(x^{\prime}-2x)=0,
\end{eqnarray}
and $\Phi_{\lambda_0
l}^{\prime}(x_1)|_{x_1=\frac{x^{\prime}}{2}}\leq\Phi_{\lambda_0
r}^{\prime}(x_1)|_{x_1=\frac{x^{\prime}}{2}}=0$.

By (\ref{b6}), (\ref{b7}) and the fact that if a convex function $f:
\mathbb{R}\rightarrow\mathbb{R}$ satisfies $f^{\prime}_r(x_0)\geq0$
and $f^{\prime}_l(x_0)\leq 0$ then $f(x_0)=\min\{ f(x):
x\in\mathbb{R}\}$, we have
\begin{eqnarray}\label{b4}
\inf_{x_1\in\mathbb{R}}\Phi_{\lambda_0}(x_1)=\Phi_{\lambda_0}(\frac{x^{\prime}}{2})=f(x^{\prime}).
\end{eqnarray}
By (\ref{b7}) and (\ref{b4}), we have
\begin{eqnarray}\label{d34}
Sf(x)=\sup_{\lambda\in [0,1]}G_x(\lambda)\geq
G_x(\lambda_0)=\inf_{x_1\in\mathbb{R}}\Phi_{\lambda_0}(x_1)=f(x^{\prime}).
\end{eqnarray}
Thus, we have $Sf(x)=f(x^{\prime})=f(x^{\prime}-2x)$.

{\bf Second Step.}  We shall prove that $Sf$ is coercive and convex,
and for any $s\in\mathbb{R}$, $Vol_1 ([Sf\leq s])=Vol_1([f\leq s])$.

First, we prove that $Sf$ is coercive. Suppose that there exists
$M_0>0$ and a sequence $\{x_n\}$ satisfying $|x_n|>n$ and
$Sf(x_n)<M_0$ for any positive integer $n$, then by (\ref{26}),
there exists $x_n^{\prime}$ such that
\begin{eqnarray}\label{c8}
Sf(x_n)=f(x_n^{\prime})=f(x_n^{\prime}-2x_n)<M_0.
\end{eqnarray}
Since $2\max\{|x_n^{\prime}|,|x_n^{\prime}-2x_n|\}\geq
|x_n^{\prime}|+|x_n^{\prime}-2x_n|\geq 2|x_n|>2n$, there is a
sequence $\{y_n\}$, where $y_n=x_n^{\prime}$ if $|x_n^{\prime}|\geq
|x_n^{\prime}-2x_n|$ and $y_n=x_n^{\prime}-2x_n$ if
$|x_n^{\prime}|\leq |x_n^{\prime}-2x_n|$, satisfying
$\lim_{n\rightarrow+\infty}|y_n|=+\infty$ and $f(y_n)<M_0$, which is
contradictory with $f$ is coercive.

Next, we prove that $Sf$ is a convex function on $\mathbb{R}$.
First, we prove that $Sf(x)$ is increasing on $[0,+\infty)$. In
fact, by (\ref{26}), for any $0< x_1<x_2$, there exist
$x_1^{\prime}$ and $x_2^{\prime}$ such that
$Sf(x_i)=f(x_i^{\prime})=f(x_i^{\prime}-2x_i)$ $(i=1,2)$. By Lemma
\ref{b14}(iii), for $\mu$ and $\nu$ given in Lemma \ref{b14}, we
have $x_i^{\prime}>\mu$ $(i=1,2)$ and $x_i^{\prime}-2x_i<\nu$
$(i=1,2)$. If $f(x_1^{\prime})> f(x_2^{\prime})$, since $f$ is
increasing on the interval $[\mu,+\infty)$, then
$x_1^{\prime}>x_2^{\prime}$. By $0<x_1<x_2$, we have
$x_1^{\prime}-2x_1>x_2^{\prime}-2x_2$. Since $f$ is decreasing on
the interval $(-\infty,\nu]$, we have $f(x_1^{\prime}-2x_1)\leq
f(x_2^{\prime}-2x_2)$, which is a contradiction. The contradiction
means that $f(x_1^{\prime})\leq f(x_2^{\prime})$, thus $Sf$ is
increasing on $[0,+\infty)$. Since $Sf$ is even, to prove $Sf$ is
convex on $\mathbb{R}$, it suffices to prove that $Sf$ is convex on
$[0,+\infty)$.

For any $0\leq x_1<x_2$ and $0<\alpha<1$, by (\ref{26}), let
$x_1^{\prime}$, $x_2^{\prime}$ and $x_0\triangleq (\alpha
x_1+(1-\alpha) x_2)^{\prime}$ be three real numbers satisfying
\begin{eqnarray}\label{a61}
Sf(x_1)=f(x^{\prime}_1)=f(x_1^{\prime}-2x_1),
\end{eqnarray}
\begin{eqnarray}\label{d51}
Sf(x_2)=f(x^{\prime}_2)=f(x_2^{\prime}-2x_2),
\end{eqnarray}
\begin{eqnarray}
\label{a62}Sf(\alpha x_1+(1-\alpha)
x_2)=f(x_0)=f(x_0-2(\alpha x_1+(1-\alpha) x_2)).
\end{eqnarray}
 Since $f$ is a convex
function, we have
\begin{eqnarray}\label{a63}\alpha f(x_1^{\prime})+(1-\alpha)f(x_2^{\prime})\geq f(\alpha
x_1^{\prime}+(1-\alpha)x_2^{\prime}),\end{eqnarray}
\begin{eqnarray}\label{a64}&&\alpha f(x_1^{\prime}-2x_1)+(1-\alpha)f(x_2^{\prime}-2x_2)\nonumber\\
&\geq& f(\alpha
x_1^{\prime}+(1-\alpha)x_2^{\prime}-2(\alpha
x_1+(1-\alpha)x_2)).\end{eqnarray}

Since $f(x_0)=f(x_0-2(\alpha x_1+(1-\alpha) x_2))$ and both
$[x_0-2(\alpha x_1+(1-\alpha)x_2),x_0]$ and $[\alpha
x^{\prime}_1+(1-\alpha)x^{\prime}_2-2(\alpha
x_1+(1-\alpha)x_2),\alpha x^{\prime}_1+(1-\alpha)x^{\prime}_2]$ have
the same length $2(\alpha x_1+(1-\alpha)x_2)>0$, by Lemma
\ref{b14}(v), either
\begin{eqnarray}\label{a65}
f(\alpha x_1^{\prime}+(1-\alpha)x_2^{\prime})\geq
f(x_0)\end{eqnarray}or
\begin{eqnarray}\label{a66}
&&f(\alpha x_1^{\prime}+(1-\alpha)x_2^{\prime}-2(\alpha
x_1+(1-\alpha)x_2))\nonumber\\
&\geq& f(x_0-2(\alpha x_1+(1-\alpha)
x_2)).
\end{eqnarray}
If (\ref{a65}) holds, then we use (\ref{a63}) and if (\ref{a66})
holds, then we use (\ref{a64}), by (\ref{a61})-(\ref{a62}), $Sf$ is
a convex function.

Finally, we prove that $Vol_1 ([f\leq s])=Vol_1([Sf\leq s])$ for any
$s\in\mathbb{R}$. Since $Sf(x)$ is an even convex function,
$Sf(0)=\inf Sf$. Since $Sf(0)=\inf f$ by (\ref{c10}), thus $\inf
Sf=\inf f$. Let $a=\inf Sf=\inf f$, $(Sf)^{-1}(a)=[-\delta,\delta]$,
and $f^{-1}(a)=[\mu,\nu]$.

If $s=a$, then $Vol_1 ([f\leq s])=\nu-\mu$ and $Vol_1([Sf\leq
s])=2\delta$. Next, we prove $\nu-\mu=2\delta$. By Lemma \ref{b14},
$Sf$ is strictly decreasing on $(-\infty,-\delta)$ and strictly
increasing on $(\delta,+\infty)$, and $f$ is strictly decreasing on
$(-\infty,\mu)$ and strictly increasing on $(\nu,+\infty)$. For
$\delta\geq0$, if $\nu-\mu>2\delta$, then
$x_0:=\delta+\frac{\nu-\mu-2\delta}{2}>\delta$, thus
$Sf(x_0)>Sf(\delta)$, which is contradictory with
\begin{eqnarray}\label{c12}
Sf(x_0)&=&\sup_{\lambda\in[0,1]}\inf_{x_1\in\mathbb{R}}[\lambda
f(2x_1)+(1-\lambda)f(2x_1-2x_0)]\nonumber\\
&\leq&\sup_{\lambda\in [0,1]}[\lambda
f(\nu)+(1-\lambda)f(\nu-2x_0)]=a,
\end{eqnarray}
where inequality is by choosing $x_1=\frac{\nu}{2}$ and last
equality is by $\nu-2x_0=\mu$. Thus, $\nu-\mu\leq2\delta$. Thus if
$\delta=0$, then $\mu=\nu$. For $\delta>0$, by (\ref{26}), there
exists $\delta^{\prime}$ such that
$Sf(\delta)=f(\delta^{\prime})=f(\delta^{\prime}-2\delta)=a$, which
implies that $\nu-\mu\geq 2\delta$. Thus, $\nu-\mu= 2\delta$.

 If $s>a$, by Lemma
\ref{b14}, equality (\ref{26}), and $Sf$ is even, there is a unique
$x>0$ and a unique $x^{\prime}\in\mathbb{R}$ such that
$Sf(-x)=Sf(x)=s=f(x^{\prime})=f(x^{\prime}-2x)$, thus we have $Vol_1
([f\leq s])=Vol_1([Sf\leq s])=2x$.

 If $s<a$, then $[Sf\leq s]=[f\leq
s]=\emptyset$, thus $Vol_1 ([f\leq s])=Vol_1([Sf\leq s])=0$.\\

{\bf Case (2) ${\rm dom}f\neq \mathbb{R}$.} There exist eight cases
for ${\rm dom}f$: 1) $[\alpha,\beta]$; 2) $(\alpha,\beta)$; 3)
$(\alpha,\beta]$; 4) $[\alpha,\beta)$; 5) $(-\infty,\beta]$; 6)
$(-\infty,\beta)$; 7) $[\alpha,+\infty)$; 8) $(\alpha,+\infty)$. We
need only prove our conclusion for ${\rm dom}f=(\alpha,\beta)$. By
the same method we can prove our conclusion for other cases. For
${\rm dom}f=(\alpha,\beta)$, there exist three cases: (i) $f$ is
decreasing on $(\alpha,\beta)$; (ii) $f$ is increasing on
$(\alpha,\beta)$; (iii) $f$ is decreasing on $(\alpha,\gamma]$ and
increasing on $[\gamma,\beta)$ for some $\gamma\in(\alpha,\beta)$.
Cases (i) and (ii) are corresponding to the cases of
$\lim_{\gamma\rightarrow\beta,\gamma<\beta}\gamma$ and
$\lim_{\gamma\rightarrow\alpha,\gamma>\alpha}\gamma$ in case (iii),
respectively, thus we need only prove our conclusion for case (iii).

If
$\lim_{x\rightarrow\alpha,x>\alpha}f(x)=\lim_{x\rightarrow\beta,x<\beta}f(x)$,
following the proof of Case (1) (i.e., ${\rm dom}f=\mathbb{R}$), we
have that $Sf$ is convex on
$(-\frac{\beta-\alpha}{2},\frac{\beta-\alpha}{2})$ and
$Vol_1([Sf\leq s])=Vol_1([f\leq s])$ for any
$s<\lim_{x\rightarrow\alpha,x>\alpha}f(x)$.

If
$\lim_{x\rightarrow\alpha,x>\alpha}f(x)\neq\lim_{x\rightarrow\beta,x<\beta}f(x)$,
we may assume that
\begin{eqnarray}\label{c15}\lim_{x\rightarrow\alpha,x>\alpha}f(x)=b>\lim_{x\rightarrow\beta,x<\beta}f(x)=c.\end{eqnarray}
If $c=a=\inf f$, then $f$ is decreasing on $(\alpha, \beta)$. Thus
we may suppose that $c>a$. Let $\gamma\in(\alpha,\beta)$ satisfy
$f(\gamma)=c$. If $|x|<\frac{\beta-\gamma}{2}$, by the proof of Case
(1), there exists $x^{\prime}\in(\gamma,\beta)$ such that
$Sf(x)=f(x^{\prime})=f(x^{\prime}-2x)$.

{\bf Step 1.} We shall prove that for
$|x|\geq\frac{\beta-\gamma}{2}$ and $|x|<\frac{\beta-\alpha}{2}$,
\begin{eqnarray}\label{d35}
Sf(x)=f(\beta-2|x|).
\end{eqnarray}

 Since $Sf$ is even, we may assume
$\frac{\beta-\gamma}{2}\leq x<\frac{\beta-\alpha}{2}$. For any
$\lambda\in[0,1]$,
\begin{eqnarray}\label{d36}
&&\inf_{x_1\in\mathbb{R}}[\lambda
f(2x_1)+(1-\lambda)f(2x_1-2x)]\nonumber\\
&\leq&\lambda\lim_{t\rightarrow
\beta,t<\beta}f(t)+(1-\lambda)f(\beta-2x)\nonumber\\
&=&\lambda c+(1-\lambda)f(\beta-2x).
\end{eqnarray}
Since $\frac{\beta-\gamma}{2}\leq x<\frac{\beta-\alpha}{2}$, then
$\alpha<\beta-2x\leq\gamma$. Since $f$ is decreasing on
$(\alpha,\gamma]$, thus $f(\beta-2x)\geq f(\gamma)=c$. Thus, by
(\ref{d36}), we have
\begin{eqnarray}\label{c19}Sf(x)&=&\sup_{\lambda\in [0,1]}\inf_{x_1\in\mathbb{R}^n}[\lambda
f(2x_1)+(1-\lambda)f(2x_1-2x)]\nonumber\\
&\leq&\sup_{\lambda\in [0,1]}[\lambda
c+(1-\lambda)f(\beta-2x)]=f(\beta-2x).
\end{eqnarray}
On the other hand, we prove that $Sf(x)\geq f(\beta-2x)$. Since
${\rm dom}f=(\alpha,\beta)$ and $\inf_{x_1\in\mathbb{R}}[\lambda
f(x_1)+(1-\lambda)f(x_1-2x)]=\inf f$ for $\lambda=0$ or $\lambda=1$,
we have
\begin{eqnarray}\label{d9}
Sf(x)=\sup_{\lambda\in(0,1)}\inf_{x_1\in(\alpha+2x,\beta)}[\lambda
f(x_1)+(1-\lambda)f(x_1-2x)].
\end{eqnarray}

By $b>c>a$, if $f^{-1}(a)=[\mu,\nu]$, then
$\alpha<\mu\leq\nu<\beta$, thus $f$ is strictly decreasing on
$(\alpha,\mu]$ and strictly increasing on $[\nu,\beta)$.
\begin{cla}\label{d38}
For a fixed $\beta^{\prime}\in(\nu,\beta)\cap(\alpha+2x,\beta)$,
there exists $\delta>0$ such that function
\begin{eqnarray}\label{d11}
G_x(x_1):=\lambda f(x_1)+(1-\lambda)f(x_1-2x)
\end{eqnarray}
is decreasing on $(\alpha+2x,\beta^{\prime}]$ for any
$0<\lambda<\delta$.
\end{cla}
\noindent{\it Proof of Claim \ref{d38}.} For $x_1\in
(\alpha+2x,\beta^{\prime}]$, the right derivative of $G_x(x_1)$
\begin{eqnarray}\label{c17}G_{xr}^{\prime}(x_1)
&=&\lambda
f_r^{\prime}(x_1)+(1-\lambda)f_r^{\prime}(x_1-2x)\nonumber\\
&\leq&\lambda
f_r^{\prime}(\beta^{\prime})+(1-\lambda)f_r^{\prime}(\beta^{\prime}-2x),\end{eqnarray}
where  the inequality is by the right derivative of a convex
function is increasing on the interior of its effective domain.
Since $\beta^{\prime}\in(\nu,\beta)\cap(\alpha+2x,\beta)$ and
$x\in[\frac{\beta-\gamma}{2},\frac{\beta-\alpha}{2})$, then
$\beta^{\prime}-2x\in(\alpha,\gamma+\beta^{\prime}-\beta)$, thus
$f_r^{\prime}(\beta^{\prime})>0$ and
$f_r^{\prime}(\beta^{\prime}-2x)<0$ for $f$ is strictly increasing
on $(\nu,\beta)$ and strictly decreasing on $(\alpha,\gamma]$. Thus,
by (\ref{c17}), we choose
\begin{eqnarray}\label{d12}
\delta=\frac{-f_r^{\prime}(\beta^{\prime}-2x)}{f_r^{\prime}(\beta^{\prime})-f_r^{\prime}(\beta^{\prime}-2x)},
\end{eqnarray}
then $G_{xr}^{\prime}(x_1)<0$ on $(\alpha+2x,\beta^{\prime}]$ for
any $\lambda\in(0,\delta)$. Therefore, $G_x(x_1)$ is decreasing on
$(\alpha+2x,\beta^{\prime}]$ for any $\lambda\in(0,\delta)$. \hfill
$\Box$\\
By (\ref{d9}) and Claim \ref{d38}, we have that
\begin{eqnarray}\label{d10}
Sf(x)&=&\sup_{\lambda\in(0,1)}\inf_{x_1\in(\alpha+2x,\beta)}[\lambda
f(x_1)+(1-\lambda)f(x_1-2x)]\nonumber\\
&\geq&\sup_{\lambda\in(0,\delta)}\inf_{x_1\in(\alpha+2x,\beta)}[\lambda
f(x_1)+(1-\lambda)f(x_1-2x)]\nonumber\\
&=&\sup_{\lambda\in(0,\delta)}\inf_{x_1\in[\beta^{\prime},\beta)}[\lambda
f(x_1)+(1-\lambda)f(x_1-2x)]\nonumber\\
&\geq&\sup_{\lambda\in(0,\delta)}[\lambda
f(\beta^{\prime})+(1-\lambda)f(\beta-2x)]\nonumber\\
&=&f(\beta-2x),
\end{eqnarray}
where the second inequality is by
$x_1\in[\beta^{\prime},\beta)\subset(\nu,\beta)$ and
$\beta^{\prime}-2x\leq x_1-2x<\beta-2x\leq\gamma$ and $f$ is
strictly increasing on $(\nu,\beta)$ and strictly decreasing on
$(\alpha,\gamma]$, and the last equality is by $f(\beta-2x)\geq
f(\beta^{\prime})$.

{\bf Step 2.} We shall prove that $Sf$ is convex in $\mathbb{R}$.
Since $Sf$ is increasing on $[0,\frac{\beta-\alpha}{2})$ and $Sf$ is
even on $(-\frac{\beta-\alpha}{2},\frac{\beta-\alpha}{2})$. Thus, it
suffices to prove $Sf$ is convex in $[0,\frac{\beta-\alpha}{2})$.
For any $x_1, x_2\in
[\frac{\beta-\gamma}{2},\frac{\beta-\alpha}{2})$ and $\lambda\in
(0,1)$, by (\ref{d35}) and $f$ is convex function, then
\begin{eqnarray}\label{c20}
\lambda Sf(x_1)+(1-\lambda)Sf(x_2)
&=&\lambda f(\beta-2x_1)+(1-\lambda)f(\beta-2x_2)\nonumber\\
&\geq&f(\beta-2(\lambda x_1+(1-\lambda)x_2))\nonumber\\
&=&Sf(\lambda x_1+(1-\lambda)x_2),
\end{eqnarray}
where the last equality is by $\lambda
x_1+(1-\lambda)x_2\in[\frac{\beta-\gamma}{2},\frac{\beta-\alpha}{2})$.
By (\ref{c20}), $Sf$ is convex on
$[\frac{\beta-\gamma}{2},\frac{\beta-\alpha}{2})$. Because that $Sf$
is convex in $[0,\frac{\beta-\gamma}{2}]$ by the proof in Case (1),
it suffices to prove that the left derivative of $Sf$ at
$x=\frac{\beta-\gamma}{2}$ is less than its right derivative at
$x=\frac{\beta-\gamma}{2}$.

By (\ref{d35}), we have
\begin{eqnarray}\label{c22}
Sf_r^{\prime}(\frac{\beta-\gamma}{2})&=&\lim_{t\rightarrow
0,t>0}\frac{Sf(\frac{\beta-\gamma}{2}+t)-Sf(\frac{\beta-\gamma}{2})}{t}\nonumber\\
&=&\lim_{t\rightarrow
0,t>0}\frac{f(\gamma-2t)-f(\gamma)}{t}=-2f_l^{\prime}(\gamma).
\end{eqnarray}
For any $t\in(-\frac{\beta-\gamma}{2},0)$, we have
$\frac{\beta-\gamma}{2}+t\in(0,\frac{\beta-\gamma}{2})$. Thus there
exist $x^{\prime},\;x^{\prime\prime}\in(\gamma,\beta)$ such that
$x^{\prime\prime}-x^{\prime}=2(\frac{\beta-\gamma}{2}+t)$ and
$Sf(\frac{\beta-\gamma}{2}+t)=f(x^{\prime})=f(x^{\prime\prime})$.
Since
\begin{eqnarray*}\label{d13}
(x^{\prime}-\gamma)+2(\frac{\beta-\gamma}{2}+t)=(x^{\prime}-\gamma)+(x^{\prime\prime}-x^{\prime})=x^{\prime\prime}-\gamma<\beta-\gamma,
\end{eqnarray*}
 $x^{\prime}<\gamma-2t$. Let $|t|$ be sufficiently small such
that $\gamma+2|t|<\mu$, where $\mu$ satisfies $f^{-1}(a)=[\mu,\nu]$,
then $f(x^{\prime})>f(\gamma-2t)$ for $f$ is strictly decreasing on
$(\gamma,\mu)$. Then
\begin{eqnarray}\label{c24}
Sf_l^{\prime}(\frac{\beta-\gamma}{2})&=&\lim_{t\rightarrow
0,t<0}\frac{Sf(\frac{\beta-\gamma}{2}+t)-Sf(\frac{\beta-\gamma}{2})}{t}
=\lim_{t\rightarrow
0,t<0}\frac{f(x^{\prime})-f(\gamma)}{t}\nonumber\\
&\leq&\lim_{t\rightarrow
0,t<0}\frac{f(\gamma-2t)-f(\gamma)}{t}=-2f_r^{\prime}(\gamma).
\end{eqnarray}
Since $f$ is a convex function, then $f_l^{\prime}(\gamma)\leq
f_r^{\prime}(\gamma)$, by (\ref{c22}) and (\ref{c24}), we have
$Sf_l^{\prime}(\frac{\beta-\gamma}{2})\leq
Sf_r^{\prime}(\frac{\beta-\gamma}{2})$.

{\bf Step 3.} Proof of $Vol_1([Sf\leq a])=Vol([f\leq s])$ for any
$s\in\mathbb{R}$.

If $s<c$, the proof is the same as in Case (1).

If $c\leq s<b$, since $f$ is strictly decreasing on
$(\alpha,\gamma)$, there is a unique $x^{\prime}\in(\alpha,\gamma)$
such that $f(x^{\prime})=s$, thus $[f<s]=[x^{\prime},\beta)$. By
(\ref{d35}), we have
$Sf(\frac{\beta-x^{\prime}}{2})=f(x^{\prime})=s$, thus
$[Sf<s]=[-\frac{\beta-x^{\prime}}{2},\frac{\beta-x^{\prime}}{2}]$.
Therefore, $Vol_1([Sf<s])=Vol_1([f<s])=\beta-x^{\prime}$.

If $s\geq b$, then $b<+\infty$ for $s\in\mathbb{R}$, we have
$Vol_1([Sf<s])=Vol_1([f<s])=\beta-\alpha$.\hfill $\Box$\\

\begin{Rem}{\bf 1)} By Theorem \ref{b10}, for any
$x\in\mathbb{R}$, if $x=0$, then $Sf(0)=\inf f$; if $x\neq0$, then
there exist three cases:
\begin{itemize}
\item[i)\;\;]$Sf(x)=f(x^{\prime})=f(x^{\prime}-2|x|)$
for some $x^{\prime}\in\mathbb{R}$;
\item[ii)\;] $Sf(x)=f(x_0-2|x|)$ for some $x_0\in\mathbb{R}$;
\item[iii)] $S_u f(x)=f(x_0+2|x|)$ for some $x_0\in\mathbb{R}$.
\end{itemize}

{\bf 2)} In Theorem \ref{b10}, there exist three cases for ${\rm
dom}Sf$: i) ${\rm dom}Sf=(-\delta,\delta)$; ii) ${\rm
dom}Sf=[-\delta,\delta]$; iii) ${\rm dom}Sf=\mathbb{R}$. ${\rm
dom}Sf=(-\delta,\delta)$ is corresponding to ${\rm
dom}f=(\alpha,\beta)$, ${\rm dom}f=(\alpha,\beta]$ or ${\rm
dom}f=[\alpha,\beta)$, where $\delta=\frac{\beta-\alpha}{2}$. ${\rm
dom}Sf=[-\delta,\delta]$ is corresponding to ${\rm
dom}f=[\alpha,\beta]$. ${\rm dom}Sf=\mathbb{R}$ is corresponding to
${\rm dom}f=(-\infty,\beta)$, ${\rm dom}f=(-\infty,\beta]$, ${\rm
dom}f=(\alpha,+\infty)$, ${\rm dom}f=[\alpha,+\infty)$ or ${\rm
dom}f=\mathbb{R}$. For a non-empty convex set $K\subset\mathbb{R}^n$
and a hyperplane $H=u^{\perp}$, where $u\in S^{n-1}$, the {\it
Steiner symmetrization} $S_H K$ of $K$ about $H$ is defined as
$$S_{H}K=\{x^{\prime}+\frac{1}{2}(t_1-t_2)u:\;x^{\prime}\in
P_H(K),\;t_i\in I_K(x^{\prime})\;{\rm for}\;i=1,2\},$$ where
$P_H(K)=\{x^{\prime}\in H:\;x^{\prime}+tu\in K\; {\rm
for\;some}\;t\in\mathbb{R}\}$ is the projection of $K$ onto $H$ and
$I_K(x^{\prime})=\{t\in\mathbb{R}:\;x^{\prime}+tu\in K\}$. By the
above definition and Definition \ref{15}, for coercive convex
function $f:\mathbb{R}^n\rightarrow\mathbb{R}\cup\{+\infty\}$ and
its Steiner symmetrization $S_u f$, we have
\begin{eqnarray}\label{c32}
{\rm dom}(S_{u^{\perp}} f)=S_{u^{\perp}} ({\rm dom}f).
\end{eqnarray}
We know that ${\rm dom}f$ is convex if $f$ is convex and the Steiner symmetrization of a non-empty convex set
 is still a convex set, thus by (\ref{c32}), ${\rm dom}(S_{u^{\perp}} f)$ is a convex set.

{\bf 3)} For a convex function
$f:\mathbb{R}^n\rightarrow\mathbb{R}\cup\{+\infty\}$, the {\it
epigraph} of $f$ is defined as ${\rm
epi}f:=\{(x,y)\in\mathbb{R}^{n+1}: x\in {\rm dom}f,\;y\geq f(x)\}$.
By the definition of epigraph and Theorem \ref{b10}, for
one-dimensional coercive convex function $f:\mathbb{R}\rightarrow
\mathbb{R}\cup\{+\infty\}$, we have ${\rm cl}({\rm
epi}Sf)=S_{e^{\perp}}({\rm cl}({\rm epi}f))$, where $e$ is a unit
vector along the $x$-axis and ${\rm cl} A$ denotes the closure of a
subset $A\subset\mathbb{R}^n$. Let $f:\mathbb{R}^n\rightarrow
\mathbb{R}\cup\{+\infty\}$ be a coercive and convex function and
$u\in S^{n-1}$. For any $x^{\prime}\in u^{\perp}$ and
$t\in\mathbb{R}$, if $\tilde{f}(t)=f(x^{\prime}+tu)$ is considered
as a one-dimensional function about $t$, then $S\tilde{f}(t)=S_u
f(x^{\prime}+tu)$. By Theorem \ref{b10}, ${\rm cl}({\rm
epi}(S\tilde{f}))=S_{e^{\perp}}({\rm cl}({\rm epi}\tilde{f}))$.
Since $x^{\prime}\in u^{\perp}$ is arbitrary, thus we have
\begin{eqnarray}\label{c28} {\rm cl}({\rm epi}(S_u f))=S_{\tilde{u}^{\perp}}({\rm cl}({\rm
epi}f)),\end{eqnarray}where
$\tilde{u}^{\perp}\subset\mathbb{R}^{n+1}$ denotes the hyperplane
through the origin and orthogonal to the unit vector
$\tilde{u}=(u,0)\in\mathbb{R}^{n+1}$.
\end{Rem}
Next, by Definition \ref{15} and Theorem \ref{b10},  we shall prove
five propositions which are corresponding to properties 1-5 in Table
1.

The following lemma is an obvious fact, and we omit its proof.
\begin{Lem1}\label{c37} For $f:\mathbb{R}^n\rightarrow\mathbb{R}\cup\{+\infty\}$, let $u\in S^{n-1}$ and $H=u^{\perp}$, if
\begin{itemize}
\item[i)\;\;]$f$ is symmetric with respect to hyperplane $H$, i.e., for any
$x^{\prime}\in H$ and $t\in\mathbb{R}$,
$f(x^{\prime}+tu)=f(x^{\prime}-tu)$;
\item[ii)\;] for any $x^{\prime}\in H$ and $t_1$, $t_2\in\mathbb{R}$, if $|t_1|\leq|t_2|$, then
$f(x^{\prime}+t_1 u)\leq f(x^{\prime}+t_2u)$;
\item[iii)]  $f$ is convex on half-space $H^{+}:=\{x^{\prime}+tu: x^{\prime}\in u^{\perp},t\geq
0\}$.
\end{itemize}
Then $f$ is a convex function on $\mathbb{R}^n$.
\end{Lem1}

\begin{pro}\label{b16}
If $f:\mathbb{R}^n\rightarrow \mathbb{R}\cup\{+\infty\}$ is a
coercive convex function and $u\in S^{n-1}$, then $S_uf$ is a
coercive convex function and symmetric about $u^{\perp}$.
\end{pro}
\begin{proof}
It is clear that $S_uf$ is symmetric about $u^{\perp}$. Indeed, for
any $x^{\prime}\in u^{\perp}$ and $t\in\mathbb{R}$, if we consider
$S_uf(x^{\prime}+tu)$ as a one-dimensional function about $t$, then
by Theorem \ref{b10} and Definition \ref{15}, we have
$S_uf(x^{\prime}+tu)=S_uf(x^{\prime}-tu)$.

{\bf Step 1.} We shall prove that $S_u f$ is coercive.

Suppose that there exist $M_0>0$ and a sequence
$\{x_n\}_{n=1}^{\infty}\subset\mathbb{R}^n$ satisfying that
$|x_n|>n$ and $S_u f(x_n)<M_0$. Next, we shall construct a sequence
$\{y_n\}$ satisfying $|y_n|>n$ and $f(y_n)<M_0$, which is
contradictory with $f$ is coercive.

For any positive integer $n\geq 1$, let $x_n=x_n^{\prime}+t_n u$ and
$x_n^{\prime}\in u^{\perp}$. There exist two cases of $t_n\neq 0$
and $t_n=0$.

(1) If $t_n\neq0$, then by Theorem \ref{b10}, there exist three
cases:

i)
$S_uf(x_n)=f(x_n^{\prime}+t_n^{\prime}u)=f(x_n^{\prime}+(t_n^{\prime}-2t_n)u)$
for some $t_n^{\prime}\in\mathbb{R}$;

ii) $S_u f(x_n)=f(x_n^{\prime}+(t_0-2t_n)u)$ for some
$t_0\in\mathbb{R}$;

iii) $S_u f(x_n)=f(x_n^{\prime}+(t_0+2t_n)u)$ for some
$t_0\in\mathbb{R}$.

For case i), since $|t_n^{\prime}|+|t_n^{\prime}-2t_n|\geq 2|t_n|$,
then either $|t_n^{\prime}|\geq |t_n|$ or
$|t_n^{\prime}-2t_n|\geq|t_n|$. If $|t_n^{\prime}|\geq |t_n|$, let
$y_n=x_n^{\prime}+t_n^{\prime}u$, then $S_uf(x_n)=f(y_n)$ and
$|y_n|=|x_n^{\prime}|+|t_n^{\prime}|\geq
|x_n^{\prime}|+|t_n|=|x_n|$. If $|t_n^{\prime}-2t_n|\geq|t_n|$, let
$y_n=x_n^{\prime}+(t_n^{\prime}-2t_n)u$, then $S_uf(x_n)=f(y_n)$ and
$|y_n|=|x_n^{\prime}|+|t_n^{\prime}-2t_n|\geq
|x_n^{\prime}|+|t_n|=|x_n|$. Since $|x_n|>n$, we have $|y_n|> n$ and
$f(y_n)=S_uf(x_n)<M_0$.

For case ii),  since $|t_0|+|t_0-2t_n|\geq 2|t_n|$, we have either
$|t_0|\geq |t_n|$ or $|t_0-2t_n|\geq|t_n|$. If
$|t_0-2t_n|\geq|t_n|$, let $y_n=x_n^{\prime}+(t_0-2t_n)u$, then
$S_uf(x_n)=f(y_n)$ and $|y_n|\geq|x_n|$. If $|t_0|\geq |t_n|$, let
$y_n=x_n^{\prime}+t_0u$ if $x_n^{\prime}+t_0u\in {\rm dom} f$,
otherwise let $y_n=x_n^{\prime}+t_0^{\prime}u$, where $t_0^{\prime}$
satisfies $x_n^{\prime}+t_0^{\prime}u\in {\rm dom}f$,
$|x_n^{\prime}+t_0^{\prime}u|>n$ and
$f(x_n^{\prime}+t_0^{\prime}u)<f(x_n^{\prime}+(t_0-2t_n)u)$, which
can be satisfied for $\lim_{t\rightarrow
t_0,\;t<t_0}f(x_n^{\prime}+t u)\leq f(x_n^{\prime}+(t_0-2t_n)u)$ by
Theorem \ref{b10}. Thus, we have $|y_n|> n$ and $f(y_n)< M_0$.

For case iii), we can construct $\{y_n\}$ with the same method as in
case (ii).

(2) If $t_n=0$, by Definition \ref{15}, we have
$Sf(x_n)=\inf_{t\in\mathbb{R}}f(x_n^{\prime}+tu)$. Since
$S_uf(x_n)<M_0$, there exists $y_n=x_n^{\prime}+t^{\prime}u$ such
that $f(y_n)<M_0$. Since $|y_n|=|x_n^{\prime}|+|t^{\prime}|\geq
|x_n^{\prime}|=|x_n|$, we have $|y_n|>n$ and $f(y_n)<M_0$.

{\bf Step 2.}  We shall prove that $S_uf$ is convex.

\begin{cla}\label{d48}
$S_uf$ is proper, i.e., $[S_uf=+\infty]\neq\mathbb{R}^n$ and
$[S_uf=-\infty]=\emptyset$.
\end{cla}
\noindent {\it Proof of Claim \ref{d48}.}  For any
$x\in\mathbb{R}^n$, let $x=x^{\prime}+tu$, where $x^{\prime}\in
u^{\perp}$. Since $f$ is a coercive convex function defined on
$\mathbb{R}^n$, one dimensional function $f(x^{\prime}+tu)$ about
$t\in\mathbb{R}$ either is a coercive convex function or is
identically $+\infty$. If $f(x^{\prime}+tu)$ is a coercive convex
function, then there exists $s\in\mathbb{R}$ such that
$s=\inf\{f(x^{\prime}+tu):t\in\mathbb{R}\}$. Thus, we have
$$Sf(x)=\sup_{\lambda\in[0,1]}\inf_{t_1+t_2=t}[\lambda
f(x^{\prime}+2t_1u)+(1-\lambda)f(x^{\prime}-2t_2u)]\geq s,$$
 which
implies that $Sf(x)>-\infty$. If $f(x^{\prime}+tu)$ is identically
$+\infty$, then $S_uf(x)=+\infty>-\infty$. By the definition of
convex functions, $f$ is not identically $+\infty$, there exists
$x\in\mathbb{R}^n$ such that $f(x)<+\infty$. Let $x=x_0+tu$, where
$x_0\in u^{\perp}$, then
$$S_uf(x_0)=\inf_{t_1\in\mathbb{R}}f(x_0+t_1u)\leq f(x)<+\infty,$$
which implies that $S_uf$ is not identically $+\infty$.\hfill
$\Box$\

By Definition {\ref{15}} and Theorem \ref{b10}, for any
$x^{\prime}\in u^{\perp}$, one-dimensional function $S_u
f(x^{\prime}+tu)$ is either an even and coercive convex function
about $t\in\mathbb{R}$ or identically $+\infty$. Thus, $S_uf$
satisfies conditions i) and ii) in Lemma \ref{c37}. Therefore, to
prove that $S_uf$ is convex, it suffices to prove that $S_uf$
satisfies condition iii) of Lemma \ref{c37}. For any $x$,
$y\in\{x^{\prime}+tu:x^{\prime}\in u^{\perp},t\geq0\}$ and
$\lambda\in (0,1)$, if $x\notin {\rm dom} (S_uf)$ or $y\notin {\rm
dom} (S_uf)$, then $S_uf(x)=+\infty$ or $S_uf(y)=+\infty$, thus
\begin{eqnarray}\label{d20}
S_uf(\lambda x+(1-\lambda)y)\leq \lambda S_uf(x)+(1-\lambda)S_uf(y).
\end{eqnarray}
 By Remark 2), ${\rm dom}(S_uf)$ is convex. Therefore, if $x\in {\rm dom}
(S_uf)$ and $y\in {\rm dom} (S_uf)$, then $\lambda x+(1-\lambda)y\in
{\rm dom}(S_uf)$. Let $x=x^{\prime}+tu$ and $y=y^{\prime}+su$, where
$x^{\prime}$, $y^{\prime}\in u^{\perp}$ and $t\geq0$ and $s\geq0$,
then $\lambda x+(1-\lambda)y=[\lambda
x^{\prime}+(1-\lambda)y^{\prime}]+[\lambda t+(1-\lambda)s]u$.

{\bf Case 3.1.} The case of $t=0$ and $s=0$. For the case we have
$x$, $y\in u^{\perp}$, thus $\lambda x+(1-\lambda)y\in u^{\perp}$.
By Definition \ref{15} and $f$ is convex, we have
\begin{eqnarray}&&\lambda
S_uf(x)+(1-\lambda)S_uf(y)\nonumber\\
&=&\lambda\inf_{t\in\mathbb{R}}f(x+tu)+(1-\lambda)\inf_{s\in\mathbb{R}}f(y+su)\nonumber\\
&=&\inf_{(t,s)\in\mathbb{R}^2}[\lambda f(x+tu)+(1-\lambda)f(y+su)]\nonumber\\
&\geq&\inf_{(t,s)\in\mathbb{R}^2}f(\lambda x+(1-\lambda)y+(\lambda t+(1-\lambda)s)u)\nonumber\\
&=&S_uf(\lambda x+(1-\lambda)y).
\end{eqnarray}

{\bf Case 3.2.} The case of $t>0$ and $s>0$.

For $x=x^{\prime}+tu\in{\rm dom}(S_uf)$, by Theorem \ref{b10}, there
exist three cases:

$a_1$) There exists some $t^{\prime}\in\mathbb{R}$ such that
\begin{eqnarray}\label{c43}S_uf(x)=f(x^{\prime}+t^{\prime}u)=f(x^{\prime}+(t^{\prime}-2t)u);\end{eqnarray}

$a_2$) There exists some $t_0\in\mathbb{R}$ such that
\begin{eqnarray}\label{c44}S_uf(x)=f(x^{\prime}+(t_0-2t)u)\geq\lim_{t_0^{\prime}\rightarrow t_0,t_0^{\prime}<t_0}f(x^{\prime}+t_0^{\prime}u);\end{eqnarray}

$a_3$) There exists some $t_0\in\mathbb{R}$ such that
\begin{eqnarray}\label{c45}S_uf(x)=f(x^{\prime}+(t_0+2t)u)\geq\lim_{t_0^{\prime}\rightarrow t_0,t_0^{\prime}>t_0}f(x^{\prime}+t_0^{\prime}u).\end{eqnarray}

For $y=y^{\prime}+su\in{\rm dom}(S_uf)$, by Theorem \ref{b10}, there
exist three cases:

$b_1$) There exists some $s^{\prime}\in\mathbb{R}$ such that
\begin{eqnarray}\label{c46}S_uf(y)=f(y^{\prime}+s^{\prime}u)=f(y^{\prime}+(s^{\prime}-2s)u);\end{eqnarray}

$b_2$) There exists some $s_0\in\mathbb{R}$ such that
\begin{eqnarray}\label{c47}S_uf(y)=f(y^{\prime}+(s_0-2s)u)\geq\lim_{s_0^{\prime}\rightarrow s_0, s_0^{\prime}<s_0}f(y^{\prime}+s_0^{\prime}u);\end{eqnarray}

$b_3$) There exists some $s_0\in\mathbb{R}$ such that
\begin{eqnarray}\label{c48}S_uf(y)=f(y^{\prime}+(s_0+2s)u)\geq\lim_{s_0^{\prime}\rightarrow s_0,
 s_0^{\prime}>s_0}f(y^{\prime}+s_0^{\prime}u).
 \end{eqnarray}

We may assume that
\begin{eqnarray}\label{d23}
&&f(x^{\prime}+t_0 u)=\lim_{t_0^{\prime}\rightarrow
t_0,t_0^{\prime}<t_0}f(x^{\prime}+t_0^{\prime}u)\;\;\;{\rm for\;\;
case}\;\; a_2),\nonumber\\
&&f(x^{\prime}+t_0 u)=\lim_{t_0^{\prime}\rightarrow
t_0,t_0^{\prime}>t_0}f(x^{\prime}+t_0^{\prime}u)\;\;\;{\rm for\;\;
case}\;\; a_3),\nonumber\\
&&f(y^{\prime}+s_0 u)=\lim_{s_0^{\prime}\rightarrow
s_0,s_0^{\prime}<s_0}f(y^{\prime}+s_0^{\prime}u)\;\;\;{\rm for\;\;
case}\;\; b_2),\nonumber\\
&&f(y^{\prime}+s_0 u)=\lim_{s_0^{\prime}\rightarrow
s_0,s_0^{\prime}>s_0}f(y^{\prime}+s_0^{\prime}u)\;\;\;{\rm for\;\;
case}\;\; b_3).
\end{eqnarray}

Let $(\tilde{t}_1,\tilde{t}_2)$ be a pair of real numbers satisfying
\begin{equation} \label{d24}
(\tilde{t}_1,\tilde{t}_2)=\left\{ \begin{aligned}
(t^{\prime}-2t,t^{\prime})\;\;\;{\rm for\;\;case}\;\; a_1)\\
(t_0-2t,t_0)\;\;\;{\rm for\;\;case}\;\; a_2)\\
(t_0,t_0+2t)\;\;\;{\rm for\;\; case}\;\; a_3).
\end{aligned} \right.
\end{equation}

Let $(\tilde{s}_1,\tilde{s}_2)$ be a pair of real numbers satisfying
\begin{equation} \label{d40}
(\tilde{s}_1,\tilde{s}_2)=\left\{\begin{aligned}
(s^{\prime}-2s,s^{\prime})\;\;\;{\rm for\;\;case}\;\; b_1)\\
(s_0-2s,s_0)\;\;\;{\rm for\;\;case}\;\; b_2)\\
(s_0,s_0+2s)\;\;\;{\rm for\;\; case}\;\; b_3).
\end{aligned} \right.
\end{equation}

Since $f$ is convex and by (\ref{c43}-\ref{c48}), for $i=1,2$, we
have
\begin{eqnarray}\label{d25}
&&\lambda S_u f(x)+(1-\lambda)S_u f(y)\nonumber\\
&\geq&\lambda f(x^{\prime}+\tilde{t}_iu)+(1-\lambda)
f(y^{\prime}+\tilde{s}_iu)\nonumber\\
&\geq&f(\lambda
x^{\prime}+(1-\lambda)y^{\prime}+(\lambda\tilde{t}_i+(1-\lambda)\tilde{s}_i)u).
\end{eqnarray}
By (\ref{d24}) and (\ref{d40}), we have
\begin{eqnarray}\label{d26}
&&[\lambda\tilde{t}_2+(1-\lambda)\tilde{s}_2]-[\lambda\tilde{t}_1+(1-\lambda)\tilde{s}_1]\nonumber\\
&=&\lambda
(\tilde{t}_2-\tilde{t}_1)+(1-\lambda)(\tilde{s}_2-\tilde{s}_1)=2[\lambda
t+(1-\lambda)s].
\end{eqnarray}

By $\lambda x+(1-\lambda)y=\lambda
x^{\prime}+(1-\lambda)y^{\prime}+(\lambda t+(1-\lambda)s)u$ and
Definition \ref{15}, we have
\begin{eqnarray}\label{d28}
&&S_uf(\lambda x+(1-\lambda)y)\nonumber\\
&=&\sup_{\delta\in[0,1]}\inf_{\omega\in\mathbb{R}}\left[\delta
f(\lambda x^{\prime}+(1-\lambda)y^{\prime}+\omega
u)\right.\nonumber\\
&&\left.+(1-\delta)f(\lambda
x^{\prime}+(1-\lambda)y^{\prime}+(\omega-2(\lambda
t+(1-\lambda)s)u)\right]\nonumber\\
&\leq&\sup_{\delta\in[0,1]}\left[\delta f(\lambda
x^{\prime}+(1-\lambda)y^{\prime}+(\lambda\tilde{t}_2+(1-\lambda)\tilde{s}_2)
u)\right.\nonumber\\
&&\left.+(1-\delta)f(\lambda x^{\prime}+(1-\lambda)
y^{\prime}+(\lambda\tilde{t}_1+(1-\lambda)\tilde{s}_1)u)\right]\nonumber\\
&\leq& \max_{i=1,2}f(\lambda
x^{\prime}+(1-\lambda)y^{\prime}+(\lambda\tilde{t}_i+(1-\lambda)\tilde{s}_i)u)\nonumber\\
&\leq&\lambda S_u f(x)+(1-\lambda)S_u f(y),
\end{eqnarray}
where the first inequality is by choosing
$\omega=\lambda\tilde{t}_2+(1-\lambda)\tilde{s}_2$ and (\ref{d26}),
and the last inequality is by (\ref{d25}).

{\bf Case 3.3.} The case of $t=0$ and $s>0$ (or  $t>0$ and $s=0$).
In this case, there exists $t_0$ such that
\begin{eqnarray}\label{d29}
S_uf(x)=\lim_{ t\rightarrow t_0,\;
 x+tu\in{\rm dom}f}f(x+tu).
\end{eqnarray}
We may assume that
\begin{eqnarray}\label{d41}
f(x+t_0u)=\lim_{t\rightarrow t_0,\; x+tu\in{\rm dom}f}f(x+tu).
\end{eqnarray}
In the proof of Case 3.2, let $\tilde{t}_1=\tilde{t}_2=t_0$, we can
get the required inequality.
\end{proof}

\begin{pro} \label{36}Let $f:\mathbb{R}^n\rightarrow \mathbb{R}\cup\{+\infty\}$ be a coercive convex
function and $u\in S^{n-1}$, then
\begin{eqnarray}\int_{\mathbb{R}^n}e^{-(S_u f)(x)}dx= \int_{\mathbb{R}^n}e^{-f(x)}dx.\end{eqnarray}
\end{pro}
\begin{proof}
By (\ref{c28}), for any $t\in\mathbb{R}$, we have
 ${\rm cl}[S_uf<t]=S_u ({\rm cl}[f<t])$.
Since Steiner symmetrization of convex sets preserves volume,
$Vol([S_uf<t])=Vol([f<t])$. By Fubini's theorem, we have
\begin{eqnarray}\label{c90}
\int_{\mathbb{R}^n}e^{-(S_u
f)(x)}dx&=&\int_{\mathbb{R}}Vol([S_uf<t])e^{-t}dt\nonumber\\
&=&\int_{\mathbb{R}}Vol([f<t])e^{-t}dt=\int_{\mathbb{R}^n}e^{-f(x)}dx.
\end{eqnarray}
\end{proof}

\begin{Lem1}\label{38}
Let $u_1,u_2\in S^{n-1}$ and $\langle u_1, u_2\rangle=0$. If
$f:\mathbb{R}^n\rightarrow \mathbb{R}\cup\{+\infty\}$ is a coercive
convex function and $f$ is symmetric about $u_1^{\perp}$, then
$S_{u_2} f$ is symmetric about both $u_1^{\perp}$ and $u_2^{\perp}$.
\end{Lem1}
\begin{proof}
By Proposition \ref{b16}, $S_{u_2}f$ is symmetric about
$u_2^{\perp}$. Next, we prove that $S_{u_2} f$ is symmetric about
$u_1^{\perp}$. Since $\langle u_1,u_2\rangle=0$, then $u_1\in
u_2^{\perp}$ and $u_2\in u_1^{\perp}$.  For any $x^{\prime}\in
u_1^{\perp}$, let $x^{\prime}=x^{\prime\prime}+t_{x^{\prime}}u_2$,
where $x^{\prime\prime}=x^{\prime}|u_2^{\perp}$. Then
$x^{\prime\prime}=x^{\prime}-t_{x^{\prime}}u_2\in u_1^{\perp}$, thus
$x^{\prime\prime}+tu_2\in u_1^{\perp}$. Because that
$x^{\prime\prime}\in u_2^{\perp}$ and $u_1\in u_2^{\perp}$, thus
$x^{\prime\prime}+tu_1\in u_2^{\perp}$. Thus, for any $x^{\prime}\in
u_1^{\perp}$ and $t\in\mathbb{R}$, we have
\begin{eqnarray}\label{39}&~&(S_{u_2}f)(x^{\prime}+tu_1)
=(S_{u_2}f)(x^{\prime\prime}+tu_1+t_{x^{\prime}}u_2)
\nonumber\\
&=&\sup_{\lambda\in[0,1]}\inf_{t_1+t_2=t_{x^{\prime}}}[\lambda f(x^{\prime\prime}+tu_1+2t_1u_2)+(1-\lambda)f(x^{\prime\prime}+tu_1-2t_2u_2)]
\nonumber\\
&=&\sup_{\lambda\in[0,1]}\inf_{t_1+t_2=t_{x^{\prime}}}[\lambda f(x^{\prime\prime}-tu_1+2t_1u_2)+(1-\lambda)f(x^{\prime\prime}-tu_1-2t_2u_2)]
\nonumber\\
&=&(S_{u_2}f)(x^{\prime\prime}-tu_1+t_{x^{\prime}}u_2)\nonumber\\
&=&(S_{u_2}f)(x^{\prime}-tu_1),
\end{eqnarray}
where the second equality is by $f$ is symmetric about $u_1^{\perp}$
and $x^{\prime\prime}+tu_2\in u_1^{\perp}$. This completes the
proof.
\end{proof}
We say that a function $f:
\mathbb{R}^n\mapsto\mathbb{R}\cup\{+\infty\}$ is {\it unconditional}
if $f(x_1,\dots, x_n)=f(|x_1|,\dots,|x_n|)$ for every
$(x_1,\dots,x_n)\in\mathbb{R}^n$.
\begin{pro}\label{40}
 Any
coercive convex function $f:\mathbb{R}^n\rightarrow
\mathbb{R}\cup\{+\infty\}$ can be transformed into an unconditional
function $\bar{f}$ using $n$ Steiner symmetrizations.
\end{pro}
\begin{proof}
Let $\{u_1,\dots, u_n\}$ be an orthonormal basis of $\mathbb{R}^n$.
By Proposition \ref{b16} and Lemma \ref{38}, $S_{u_n}\cdots S_{u_1}
f$ is symmetric about $u_i^{\bot}$, $i=1,\cdots, n$, which implies
that $f$ can be transformed into an unconditional function
$\bar{f}=S_{u_n}\cdots S_{u_1} f$ using $n$ Steiner symmetrizations.
\end{proof}
\begin{pro}\label{b18}
Let $f_1:\mathbb{R}^n\rightarrow \mathbb{R}\cup\{+\infty\}$ and
$f_2:\mathbb{R}^n\rightarrow \mathbb{R}\cup\{+\infty\}$ be coercive convex
functions and $u\in S^{n-1}$. If $f_1\leq f_2$ (which implies that
$f_1(x)\leq f_2(x)$ for any $x\in\mathbb{R}^n$), then $S_u f_1\leq
S_u f_2$.
\end{pro}
\begin{proof}
By Definition \ref{15} and $f_1\leq f_2$, for $x=x^{\prime}+tu$,
where $x^{\prime}\in u^{\perp}$, we have
\begin{eqnarray}\label{b20}
S_u f_1(x)&=&\sup_{\lambda\in
[0,1]}\inf_{t_1+t_2=t}[\lambda f_1(x^{\prime}+2t_1u)+(1-\lambda) f_1(x^{\prime}-2t_2u)]
\nonumber\\
&\leq&\sup_{\lambda\in [0,1]}\inf_{t_1+t_2=t}[\lambda
f_2(x^{\prime}+2t_1u)+(1-\lambda)f_2(x^{\prime}-2t_2u)]\nonumber\\
&=&S_u f_2(x).
\end{eqnarray}
\end{proof}
We say a function $f$ is even about point $z\in\mathbb{R}^n$ if
$f(z+x)=f(z-x)$ for any $x\in\mathbb{R}^n$. Let $z|H$ denote the
projection of $z$ onto hyperplane $H$.
\begin{pro}\label{33}
Let $f:\mathbb{R}^n\rightarrow \mathbb{R}\cup\{+\infty\}$ be a
coercive convex function and $u\in S^{n-1}$, if $f$ is even about
$z$, then $S_u f$ is even about $z|u^{\perp}$.
\end{pro}
\begin{proof}
For any $x\in\mathbb{R}^n$, let $x=x^{\prime}+tu$, where
$x^{\prime}=x|u^{\perp}$. Let $z=z^{\prime}-t_0u$, where
$z^{\prime}=z|u^{\perp}$. By Definition \ref{15}, we have
\begin{eqnarray}\label{d30}
&~&(S_u f)(z^{\prime}+x)=(S_u f)(z^{\prime}+x^{\prime}+tu)=(S_u f)(z^{\prime}+x^{\prime}-tu)\nonumber\\
&=&\sup_{\lambda\in[0,1]}\inf_{t_1+t_2=-t}[\lambda
f(z^{\prime}+x^{\prime}+2t_1u)+(1-\lambda)f(z^{\prime}+x^{\prime}-2t_2u)]
\nonumber\\
&=&\sup_{\lambda\in[0,1]}\inf_{t_2\in\mathbb{R}}[\lambda
f(z+t_0u+x^{\prime}-2t_2u-2tu)+(1-\lambda)f(z+t_0u+x^{\prime}-2t_2u)]
\nonumber\\
&=&\sup_{\lambda\in[0,1]}\inf_{t_2\in\mathbb{R}}[\lambda
f(z+x^{\prime}-2t_2u-2tu)+(1-\lambda)f(z+x^{\prime}-2t_2u)]
\nonumber\\
&=&\sup_{\lambda^{\prime}\in[0,1]}\inf_{t_2\in\mathbb{R}}[\lambda^{\prime}
f(z+x^{\prime}-2t_2u)+(1-\lambda^{\prime})f(z+x^{\prime}-2t_2u-2tu)],
\end{eqnarray}
where the second equality is by $S_uf$ is symmetric about
$u^{\perp}$ and the fifth equality is by replacing $t_0-2t_2$ by
$-2t_2$.

On the other hand, since $f$ is even about $z$, we have
\begin{eqnarray}\label{d43}
&~&(S_u f)(z^{\prime}-x)=(S_u f)(z^{\prime}-x^{\prime}-tu)\nonumber\\
&=&\sup_{\lambda\in[0,1]}\inf_{t_1+t_2=-t}[\lambda
f(z^{\prime}-x^{\prime}+2t_1u)+(1-\lambda)f(z^{\prime}-x^{\prime}-2t_2u)]
\nonumber\\
&=&\sup_{\lambda\in[0,1]}\inf_{t_1\in\mathbb{R}}[\lambda
f(z+t_0u-x^{\prime}+2t_1u)+(1-\lambda)f(z+t_0u-x^{\prime}+2t_1u+2tu)]
\nonumber\\
&=&\sup_{\lambda\in[0,1]}\inf_{t_1\in\mathbb{R}}[\lambda
f(z-x^{\prime}+2t_1u)+(1-\lambda)f(z-x^{\prime}+2t_1u+2tu)]
\nonumber\\
&=&\sup_{\lambda\in[0,1]}\inf_{t_1\in\mathbb{R}}[\lambda
f(z+x^{\prime}-2t_1u)+(1-\lambda)f(z+x^{\prime}-2t_1u-2tu)],
\end{eqnarray}
where the last equality is by $f$ is even about $z$. By (\ref{d30})
and (\ref{d43}), we have $(S_uf)(z^{\prime}+x)=(S_uf)(z^{\prime}-x)$
for any $x\in\mathbb{R}^n$.
\end{proof}
\section{The relation between new definition and former definitions}

\subsection{ The relation between Definition \ref{15} and Definition
\ref{c26}.}\

The relation can be generalized as follows:

(i) $S_uf$ is in general larger than $\widetilde{S}_u f$ (look at
Example \ref{c63}).

(ii) For one-dimensional coercive convex function
$f:\mathbb{R}\rightarrow \mathbb{R}\cup\{+\infty\}$, if $f$ is
symmetric about an axes $x=x_0$, i.e., $f(x_0-x)=f(x_0+x)$ for any
$x\in\mathbb{R}$, then $S f=\widetilde{S}f$.

(iii) For $n$-dimensional coercive convex function
$f:\mathbb{R}^n\rightarrow \mathbb{R}\cup\{+\infty\}$ and $u\in
S^{n-1}$, if for any $x^{\prime}\in u^{\perp}$, one-dimensional
function $f(x^{\prime}+tu)$ about $t\in\mathbb{R}$ is symmetric
about an axes $t=t_0$, then $S_uf=\widetilde{S}_uf$.

\begin{example}\label{c63}
For one-dimensional coercive convex function

\begin{equation} \label{c64}
f(x)=\left\{ \begin{aligned}
x^3\;\;{\rm if}\;\;x\geq0,\\
x^2\;\;{\rm if}\;\;x\leq0.
\end{aligned} \right.
\end{equation}
We compare $Sf$ with $\widetilde{S} f$, where
$$Sf(x)=\sup_{\lambda\in[0,1]}\inf_{x_1+x_2=x}[\lambda
f(2x_1)+(1-\lambda)f(-2x_2)]$$and
$$\widetilde{S} f(x)=\inf_{x_1+x_2=x}[\frac{1}{2}
f(2x_1)+\frac{1}{2}f(-2x_2)].$$
\end{example}
By calculation, we can get that
\begin{equation} \label{c65}
\widetilde{S} f(x)=\left\{ \begin{aligned}
\frac{(-12x-1)\sqrt{1+12x}+18x+1}{27}+2x^2\;\;{\rm if}\;\;x\geq0,\\
\frac{(12x-1)\sqrt{1-12x}-18x+1}{27}+2x^2\;\;\;{\rm if}\;\;x\leq0.
\end{aligned} \right.
\end{equation}
and
\begin{eqnarray}\label{c66}
S f(x)=g^{-1}(|x|),
\end{eqnarray}
where $g^{-1}$ is the inverse function of
\begin{eqnarray}\label{c67}
g(x)=\frac{1}{2}(\sqrt[3]{x}+\sqrt{x}),\;\;x\in[0,\infty).
\end{eqnarray}
By Matlab, we can draw their figures (see Figure 1). In the figure,
we can find that the level sets of $Sf$ and $f$ have the same size
and $Sf>\widetilde{S}f$.
\begin{figure}[htb]
\centering
 \includegraphics[height=6cm]{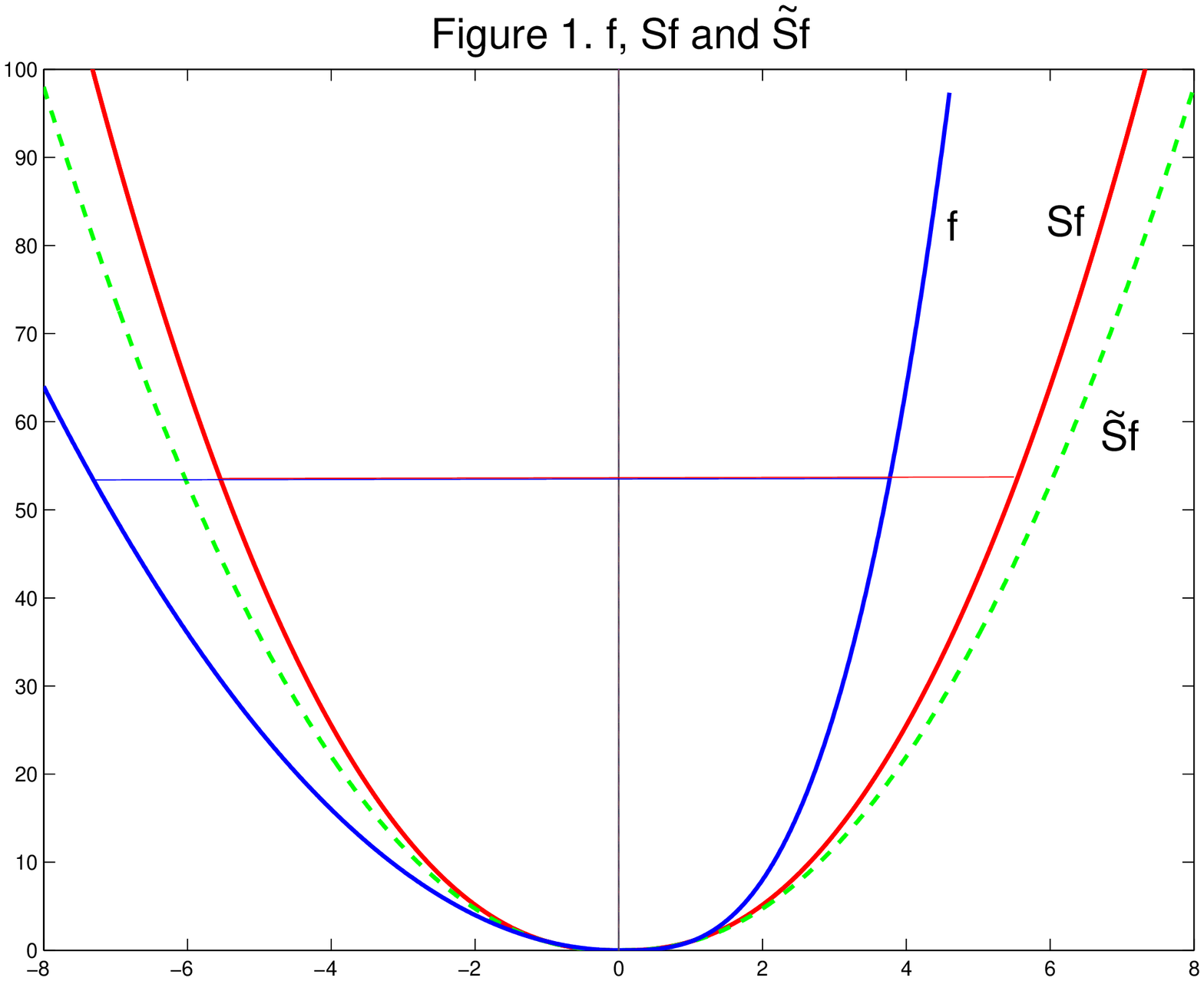}
\end{figure}
\subsection{ The relation between Definition \ref{15} and Definition
\ref{c25}.}\

In this section, we show that the two definitions are same for
log-concave functions (Theorem \ref{c72}).

\begin{lem}\label{c74}
Let $F=e^{-f}$ be a log-concave function, where
$f:\mathbb{R}^n\rightarrow\mathbb{R}$ is a coercive convex function,
then $[\bar{S}_u F>t]=S_u ([F>t])$.
\end{lem}
\begin{proof}
By Definition \ref{c25}, if $\bar{S}_u F(x)>t$, then $x\in S_u
([F>t])$. On the other hand, if $x\in S_u ([F>t])$, since $S_u
([F>t])$ is an open set and $F$ is continuous, then there exists
$t^{\prime}>t$ such that $x\in S_u ([F>t^{\prime}])$, by (\ref{1}),
we have $\bar{S}_u F(x)>t$.
\end{proof}
\begin{thm}\label{c72}
Let $f:\mathbb{R}^n\rightarrow\mathbb{R}$ be a coercive convex
function and $u\in S^{n-1}$, then $e^{(-S_u f)}=\bar{S}_u(e^{-f})$,
where $S_uf$ and $\bar{S}_u(e^{-f})$ are given in (\ref{3}) and
(\ref{1}), respectively.
\end{thm}
\begin{proof}For $t>0$, we have
\begin{eqnarray}\label{c77}
[e^{(-S_u f)}>t]=[S_uf<-\ln t]=S_u([f<-\ln t])=S_u([e^{-f}>t]),
\end{eqnarray}
where the second equality holds by (\ref{c28}).

 By Lemma
\ref{c74}, we have $[\bar{S}_u (e^{-f})>t]=S_u ([e^{-f}>t])$, thus
$[e^{(-S_u f)}>t]=[\bar{S}_u (e^{-f})>t]$. Using the ``layer-cake
representation", we have
\begin{eqnarray}\label{c80}
e^{(-S_u f)}=\int_{0}^{\infty}\mathcal {X}_{[e^{(-S_u f)}>t]}(x)dt
=\int_{0}^{\infty}\mathcal {X}_{[\bar{S}_u (e^{-f})>t]}(x)dt
=\bar{S}_u (e^{-f}).
\end{eqnarray}
\end{proof}

The continuity and convergence of Steiner symmetrization in $L^p$
space have been proved in many papers \cite{Br74,Bu97,Bu09,Fo10},
especially Proposition 3 and Theorem 2 in \cite{Fo10} are
corresponding to the properties 6-7 in Table 1.

\section{Application to functional Blaschke-Santal\'{o} inequality}

We can use the new definition to prove some important inequalities,
such as functional Blaschke-Santal\'{o} inequality,
Pr\'{e}kopa-Leindler inequality for log-concave functions,
Hardy-Littlewood inequality for log-concave functions, etc. As an
illustration, here we only use it to prove the functional
Blaschke-Santal\'{o} inequality for even convex functions.

 For a convex body $K\subset \mathbb{R}^n$, its polar about
$z$ is defined by $K^{z}=\{x\in \mathbb{R}^n: \sup_{y\in K}\langle
x-z,y-z\rangle\leq 1\}$. For a log-concave function
$f:\mathbb{R}^n\rightarrow [0,\infty)$, its polar about $z$ is
defined by
\begin{eqnarray}\label{52}
f^{z}(x)=\inf_{y\in \mathbb{R}^n}\frac{e^{-\langle
x-z,y-z\rangle}}{f(y)}.
\end{eqnarray}

To better understand this definition recall the classical Legendre
transform: For a function $\phi: \mathbb{R}^n\rightarrow
\mathbb{R}$, its Legendre transform about $z$ is defined by
$\mathcal {L}^{z}\phi(x)=\sup_{y\in\mathbb{R}^n}[\langle
x-z,y-z\rangle-\phi(y)]$. From above definition of polarity, if
$f(x)=e^{-\phi(x)}$, where $\phi(x)$ is a convex function, then
$f^{z}(x)=e^{-\mathcal {L}^{z}\phi(x)}$. Since $\mathcal
{L}^z(\mathcal {L}^z\phi)=\phi$ for a convex function $\phi$,
$(f^z)^z=f$. For $z=0$, we denote $\mathcal {L}^{0}\phi=\mathcal
{L}\phi$.

For a convex body $K$, its Santal\'{o} point $s(K)$ satisfies
$Vol(K^{s(K)})=\min_{z}Vol(K^{z})$. The Blaschke-Santal\'{o}
inequality \cite{Bl85,Sa49} states that $Vol(K)Vol(K^{s(K)})\leq
Vol(B_2^n)^2$, where $B_2^n=\{x\in\mathbb{R}^n: |x|\leq 1\}$ is the
Euclidean ball ($|\cdot|$ denote the Euclidean norm). The functional
Blaschke-Santal\'{o} inequality of log-concave functions is the
analogue of Blaschke-Santal\'{o} inequality  of convex bodies.  If
$f$ is a nonnegative integrable function on $\mathbb{R}^n$ such that
$f^{0}$ has its barycenter at $0$, then
$$\int_{\mathbb{R}^n}f(x)dx\int_{\mathbb{R}^n}f^{0}(y)dy\leq
\left(\int_{\mathbb{R}^n}e^{-\frac{1}{2}|x|^2}dx\right)^2=(2\pi)^n.$$

In the special case where the function $f$ is even, this result
follows from an earlier inequality of  Ball \cite{Ba86}; and in
\cite{Fr07}, Fradelizi and Meyer prove something more general (see
also \cite{Le09}).  Recently, Lehec \cite{Leh09} gave a direct proof
of the functional Blaschke-Santal\'{o} inequality.

 In this paper, inspired by the proof of K. Ball \cite{Ba86} for
Santal\'{o} inequality for centrally symmetric convex bodies, we
prove functional Blaschke-Santal\'{o} inequality for even convex
functions. For the non-even case, we can prove the inequality by the
similar method, but we don't prove it here.

\begin{thm}\label{c91}
(K. Ball, \cite{Ba86}) Let $f:\mathbb{R}^n\rightarrow[0,\infty)$ be
an even convex function. Assume that $0<\int e^{-f}<\infty$. Then
\begin{eqnarray}\label{d49}
\int e^{-f}\int e^{-\mathcal {L} f}\leq (2\pi)^n.
\end{eqnarray}
\end{thm}
First, we give the following lemmas.
\begin{lem}\label{c123}
Let $f:\mathbb{R}^n\rightarrow[0,\infty)$ be an even convex function
and $u\in S^{n-1}$. Assume that $0<\int e^{-f}<\infty$. Then
\begin{eqnarray}\label{d50}
\int e^{-\mathcal {L} f}\leq\int e^{-\mathcal {L} (S_uf)}.
\end{eqnarray}
\end{lem}
\begin{proof}
After a linear transformation, it may be supposed that
$H=u^{\perp}=\{(x_i)_{i=1}^{i=n}:x_n=0\}$. For $f$ and
$t\in\mathbb{R}$, we define a new function
$f_{(t)}(x^{\prime}):=f(x^{\prime}+tu)$, where $x^{\prime}\in H$.

By the definition of Steiner symmetrization, for
$x^{\prime}=x_1^{\prime}+x_2^{\prime}$, where $x^{\prime}$,
$x_1^{\prime}$ and $x_2^{\prime}\in H$, let $(x^{\prime},t)$ denote
$x^{\prime}+tu$, we have
\begin{eqnarray}\label{c97}
&&(\mathcal {L}(S_u f))_{(t)}(x^{\prime})=(\mathcal {L}(S_u f))(x^{\prime}+tu)\nonumber\\
&=&\sup_{(y^{\prime},s)\in H\times
\mathbb{R}}\left[\langle(x^{\prime},t),(y^{\prime},s)\rangle-(S_u
f)(y^{\prime}+su)\right]\nonumber\\
&=&\sup_{(y^{\prime},s)\in H\times
\mathbb{R}}[\langle(x^{\prime},t),(y^{\prime},s)\rangle-\sup_{\lambda\in[0,1]}
\inf_{s_1+s_2=s}(\lambda f(y^{\prime}+2s_1u)+(1-\lambda)f(y^{\prime}-2s_2u))]\nonumber\\
&=&\sup_{(y^{\prime},s)\in H\times
\mathbb{R}}\inf_{\lambda\in[0,1]}\sup_{s_1+s_2=s}[\langle(x^{\prime},t),
(y^{\prime},s)\rangle-(\lambda f(y^{\prime}+2s_1u)+(1-\lambda)f(y^{\prime}-2s_2u))]\nonumber\\
&\leq&\sup_{(y^{\prime},s)\in H\times
\mathbb{R}}\sup_{s_1+s_2=s}[\langle(x^{\prime},t),(y^{\prime},s)
\rangle-(\frac{1}{2} f(y^{\prime}+2s_1u)+\frac{1}{2}f(y^{\prime}-2s_2u))]\nonumber\\
&=&\sup_{(y^{\prime},s)\in H\times
\mathbb{R}}\sup_{s_1\in\mathbb{R}}[\langle(x^{\prime},t),
(y^{\prime},s)\rangle-(\frac{1}{2} f(y^{\prime}+2s_1u)+\frac{1}{2}f(y^{\prime}+2(s_1-s)u))]\nonumber\\
&\leq&\frac{1}{2}\sup_{(y^{\prime},s)\in H\times
\mathbb{R}}\sup_{s_1\in\mathbb{R}}[\langle(2x_1^{\prime},t),
(y^{\prime},2s_1)\rangle-f(y^{\prime}+2s_1u)]\nonumber\\
&~&+\frac{1}{2}\sup_{(y^{\prime},s)\in H\times
\mathbb{R}}\sup_{s_1\in\mathbb{R}}[\langle(2x_2^{\prime},-t),
(y^{\prime},2s_1-2s)\rangle-f(y^{\prime}+2(s_1-s)u)]\nonumber\\
&=&\frac{1}{2}[(\mathcal {L}f)(2x_1^{\prime}+tu)+(\mathcal
{L}f)(2x_2^{\prime}-tu)],
\end{eqnarray}
where the first inequality is by choosing $\lambda=\frac{1}{2}$ and
the second inequality is by $\sup\sup(A+B)\leq \sup\sup A+\sup\sup
B$.

Since $x_1^{\prime}$ and $x_2^{\prime}$ are arbitrary, by
(\ref{c97}), we can get
\begin{eqnarray}\label{c95}
\left(e^{-(\mathcal {L}(S_u f))_{(t)}}\right)(x^{\prime})\geq
\sup_{x_1^{\prime}+x_2^{\prime}=x^{\prime}}
\left(e^{-\frac{1}{2}(\mathcal {L}f)_{(t)}(2x_1^{\prime})}\times
e^{-\frac{1}{2}(\mathcal {L}f)_{(-t)}(2x_2^{\prime})}\right).
\end{eqnarray}
By (\ref{c95}) and Pr\'{e}kopa-Leindler inequality, we have
\begin{eqnarray}\label{c99}
\int_{H}e^{-(\mathcal {L}(S_u
f))_{(t)}}(x^{\prime})dx^{\prime}&\geq&\left(\int_{H}e^{-(\mathcal
{L}f)_{(t)}(x^{\prime})}dx^{\prime}\right)^{\frac{1}{2}}\left(\int_{H}e^{-(\mathcal
{L}f)_{(-t)}(x^{\prime})}dx^{\prime}\right)^{\frac{1}{2}}\nonumber\\&=&\int_{H}e^{-(\mathcal
{L}f)_{(t)}(x^{\prime})}dx^{\prime},
\end{eqnarray}
where the last equality is by $\mathcal {L}f$ is even (since $f$ is
even). Thus, by Fubini's theorem, we can get the desired inequality.
\end{proof}
\begin{lem}\label{c101}
Let $h(t)$ be an increasing convex function defined on $[0,+\infty)$
and $\int_{0}^{+\infty}e^{-h(t)}dt<\infty$. Let $\mathcal
{L}(h(|\cdot|))$ denote the Legendre transform of function $h(|x|)$
defined on $\mathbb{R}^n$. Then
\begin{eqnarray}\label{c102}
\int_{\mathbb{R}^n}e^{-h(|x|)}dx\int_{\mathbb{R}^n}e^{-(\mathcal
{L}(h(|\cdot|)))(x)}dx\leq (2\pi)^n.
\end{eqnarray}
\end{lem}
\begin{proof} By spherical coordinate transformation, we have
\begin{eqnarray}\label{c104}
\int_{\mathbb{R}^n}e^{-h(|x|)}dx=\int_{S^{n-1}}\left[\int_{0}^{+\infty}e^{-h(r)}r^{n-1}dr\right]d\omega.
\end{eqnarray}
For any $x\in\mathbb{R}^n$,  let $x=t_x\theta_x$, where
$\theta_x=\frac{x}{|x|}\in S^{n-1}$ for $|x|\neq0$ and $\theta_x$ is
any unit vector for $|x|=0$, and $t_x=|x|$. Then, we have
\begin{eqnarray*}\label{c105}
\mathcal {L}(h(|\cdot|))(x)&=&\sup_{y\in\mathbb{R}^n}(\langle
x,y\rangle-h(|y|))\nonumber\\
&=&\sup_{\theta_y\in S^{n-1},t_y\geq 0}(\langle
t_x\theta_x,t_y\theta_y\rangle-h(t_y))=\sup_{t_y\geq
0}(t_xt_y-h(t_y)).
\end{eqnarray*}
Thus, we have
\begin{eqnarray}\label{c106}
\int_{\mathbb{R}^n}e^{-(\mathcal
{L}(h(|\cdot|)))(x)}dx=\int_{S^{n-1}}\left[\int_{0}^{+\infty}\left(e^{-\sup_{t\geq
0}(rt-h(t))}\right)r^{n-1}dr\right]d\omega.
\end{eqnarray}
For $r\in[0,+\infty)$, let $f_1(r)=\left(e^{-h(r)}\right)r^{n-1}$,
$f_2(r)=\left(e^{-\sup_{t\geq 0}(rt-h(t))}\right)r^{n-1}$ and
$f_3(r)=\left(e^{-\frac{r^2}{2}}\right)r^{n-1}$. Next, we shall
prove that
\begin{eqnarray}\label{c110}
\int_{0}^{+\infty}f_1(r)dr\int_{0}^{+\infty}f_2(r)dr\leq\left(\int_{0}^{+\infty}f_3(r)dr\right)^2.
\end{eqnarray}
Let $g_i(t)=f_i(e^t)e^t$ for $i=1,2,3,$ then
$\int_{0}^{+\infty}f_i(r)dr=\int_{\mathbb{R}}g_i(t)dt$ and for every
$s,t\in\mathbb{R}$, $g_1(s)g_2(t)\leq
\left(g_3(\frac{s+t}{2})\right)^2$. Hence inequality (\ref{c110})
follows from Pr\'{e}kopa-Leindler inequality.

By (\ref{c104}), (\ref{c106}) and (\ref{c110}), we have
\begin{eqnarray*}\label{c113}
&&\int_{\mathbb{R}^n}e^{-h(|x|)}dx\int_{\mathbb{R}^n}e^{-(\mathcal
{L}(h(|\cdot|)))(x)}dx\nonumber\\
&=&\omega_n^2\int_{0}^{+\infty}f_1(r)dr\int_{0}^{+\infty}f_2(r)dr\leq\omega_n^2\left(\int_{0}^{+\infty}e^{-\frac{r^2}{2}}r^{n-1}dr\right)^2
=(2\pi)^n,
\end{eqnarray*}
where $\omega_n=n\pi^{n/2}/\Gamma(1+\frac{n}{2})$ is the surface
area of Euclidean unit ball.
\end{proof}

\noindent{\it Proof of Theorem \ref{c91}.} By the integral
invariance under Steiner symmetrization (Proposition \ref{36}), for
any $u\in S^{n-1}$, we have
\begin{eqnarray}\label{d44}
\int_{\mathbb{R}^n}e^{-(S_uf)(x)}dx=\int_{\mathbb{R}^n}e^{-f(x)}dx.
\end{eqnarray}
By (\ref{d44}) and Lemma \ref{c123}, we have
\begin{eqnarray}\label{c122}
\int e^{-f}\int e^{-\mathcal {L} f}\leq\int e^{-S_uf}\int
e^{-\mathcal {L}(S_uf)}.
\end{eqnarray}
By property 7 in Table {\ref{a59}}, for log-concave function
$e^{-f}\in L^1(\mathbb{R}^n)$,  there exists a sequence of
directions $\{u_i\}_{i=1}^{\infty}\subset S^{n-1}$ such that
$e^{-S_{u_1,\dots,u_i}f}$ converges to a radial function
$e^{-h(|\cdot|)}$, where $h(t)$ is a one-dimensional increasing
convex function defined on $[0,+\infty)$.
 By
(\ref{c122}) and Lemma \ref{c101} and the continuity of integral in
$L^1(\mathbb{R}^n)$, we have
\begin{eqnarray}\label{c125}
\int e^{-f}\int e^{-\mathcal {L}f} &\leq& \lim_{i\rightarrow
+\infty} \int e^{-S_{u_1,\dots,u_i}f}\int e^{-\mathcal
{L}(S_{u_1,\dots,u_i}f)}\nonumber\\
&=&\int e^{-h(|\cdot|)}\int e^{-\mathcal {L}(h(|\cdot|))}\leq
(2\pi)^n.
\end{eqnarray}
This completes the proof.\hfill $\Box$\\

\bibliographystyle{amsalpha}

\end{document}